\documentclass[leqno,12pt]{amsart}

\usepackage{amsmath}
\usepackage{amsfonts}
\usepackage{amssymb}
\usepackage{eucal}

\evensidemargin=\oddsidemargin
\newtheorem{theorem}{Theorem}
\newtheorem{corollary}{Corollary}
\newtheorem{remark}{Remark}

\newenvironment{example}[1]{\noindent{\bf Example #1}.}{\space\hfill\qed\vskip2mm}

\def\Det{\mathop{\rm Det}}
\def\sign{\mathop{\rm sign}}

\usepackage[notcite, notref]{showkeys}

\begin{document}

\title{The equi-affine curvatures of curves in 3-dimensional pseudo-Riemannian manifolds}

\author{Karina Olszak and Zbigniew Olszak}

\dedicatory{Dedicated to Professor Udo Simon on his 81st birthday}

\date{\today}



\begin{abstract}
In this paper, the Cartan frames and the equi-affine curvatures are described with the help of the Frenet frames and the Frenet curvatures of a non-null and non-degenerate curve in a 3-dimensional pseudo-Riemannian manifold. The constancy of the Frenet curvatures of such a curve always implies the constancy of the equi-affine curvatures. We show that the converse statement does not hold in general. Finally, we study the equi-affine curvatures of null curves in 3-dimensional Lorentzian manifolds, and prove that they are related to their pseudo-torsion.
\end{abstract}

\maketitle

\section{Introduction}

The studies of the curvatures of curves in equi-affine geometry have already the long history. The beginings of it we can find in the classical books about curves in affine spaces of dimension 2 and 3; see W. Blaschke \cite{Bla}, H. W. Guggenheimer \cite{Gug}, B. Su \cite{Su}, cf. also J. Favard \cite{Fav}, K. Nomizu and T. Sasaki \cite{NS}. Next, such curves were investigated also in higher dimensional affine spaces by, among others, D. Davis \cite{Dav}, S. Izumiya and T. Sano \cite{IS}, D. Khadjiev and \"O. Peksen \cite{KP}, S. Kobayashi and T. Sasaki \cite{KS}, M. Nadjafikhah and A. Mahdipour-Shirayeh \cite{NMS-1,NMS-2}, L. A. Santal\'o \cite{San}.

On the other hand, these ideas were extended to the non-degenerate curves in equi-affine manifolds in the papers by W. Barthel and A. Irmingard \cite{Bar}, M. Faghfouri and M. Toomanian \cite{FT}, V. Hlavat\'y \cite{Hla} and P. Stavre \cite{Sta}. 

The present paper is a continuation of the authors works on the equi-affine curvature of curves in pseudo-Riemannian (including Riemannian) manifolds. Namely, in the paper \cite{OOl}, we have expressed the equi-affine curvature with the help of the Frenet (geodesic) curvature of a non-degenerate curve in the case when the pseudo-Riemannian manifolds are 2-dimensional. Some motivations for study of equi-affine curvature of curves are also mentioned in \cite{OOl}.

In the presented paper, as the main achievement, we describe the Cartan frames and the equi-affine curvatures with the help of the Frenet frames and the Frenet curvatures of a non-null and non-degenerate curve in a 3-dimensional pseudo-Riemannian manifold. As it can be seen, the constancy of the Frenet curvatures of such a curve always implies the constancy of the equi-affine curvatures. We will show that the converse statement does not hold in general. Finally, we study the equi-affine curvatures for null curves in 3-dimensional Lorentzian manifolds, and we show that they are related to their pseudo-torsion.

\section{Curves in 3-dimensional equi-affine manifolds}

In this section, following the classical books and papers mentioned in the introduction, we present the notion of the equi-affine curvature of a curve in a 3-dimensional equiaffine manifold.

Let $(M^3,\nabla,\varOmega)$ be a $3$-dimensional equi-affine manifold, that is, $M^3$ is an oriented $3$-dimensional differentiable manifold endowed with an affine connection $\nabla$ and a volume element (a nowhere vanishing $3$-form) $\varOmega$ such that $\nabla\varOmega=0$. If it is not confused, such a manifold will be denoted by $M^3$.

Let $\alpha\colon I\to M^3$, $I$ being an open interval, be a non-degenerate curve in $M^3$, that is, the mapping is smooth and 
$
  \varOmega(\alpha', \nabla_{\alpha'}\alpha', \nabla^2_{\alpha'}\alpha') \neq 0
$
at every point of the curve. Thus, the vector fields $\alpha'$, $\nabla_{\alpha'}\alpha'$, $\nabla^2_{\alpha'}\alpha'$ are linearly independent at every point of the curve $\alpha$. The number 
\begin{align*}
  \varepsilon = \sign(\varOmega(\alpha',
            \nabla_{\alpha'}\alpha',
            \nabla^2_{\alpha'}\alpha')) (\, = \pm1) 
\end{align*}
will be called the orientation of $\alpha$. In the case when $M^3$ is the standard affine space $\mathbb A^3$, a curve $\alpha$ is called dextrorse (right winding) if $\varepsilon=+1$, and sinistrorse (left winding) if $\varepsilon=-1$. For the further use, we assume recursively that
$
  \nabla^{k+1}_{\alpha'}\alpha'
  = \nabla_{\alpha'}\big(\nabla^k_{\alpha'}
    \alpha'\big),
$
for any $k\geqslant1$ and $\nabla^1_{\alpha'}\alpha' = \nabla_{\alpha'}\alpha'$. 

Define the equi-affine speed function $\mu$ and an auxiliary function $\varphi$ on $I$ by  
\begin{equation}
\label{mu}
  \mu = |\varOmega(\alpha',
            \nabla_{\alpha'}\alpha',
            \nabla^2_{\alpha'}\alpha')|^{1/6}, \quad \varphi = 1/\mu. 
\end{equation}
Then, the equi-affine arc length parameter $\sigma$ starting at $t_0\in I$ is defined with the help of $\mu$ as 
\begin{align*}
  \sigma(t) = \int^t_{t_0} \mu(u)\,du, \ t \in I.
\end{align*}
However, in the sequel, we still assume that the parametrization is arbitrary. 

Define the Cartan moving frame $(\mathbf e_1,\mathbf e_2,\mathbf e_3)$ along the curve $\alpha$ by 
\begin{align}
\label{defCartfr}
  \mathbf e_1 = \varphi \alpha', \quad 
	\mathbf e_2 = \nabla_{\mathbf e_1}\mathbf e_1, \quad
	\mathbf e_3 = \nabla_{\mathbf e_1}\mathbf e_2
\end{align}

Note that then, we have 
\begin{align}
  \mathbf e_1 
    &= \varphi \alpha', \nonumber\\
\label{e1e2e3}
  \mathbf e_2 
    &= \varphi \nabla_{\alpha'}\mathbf e_1 
     = \varphi \varphi' \alpha' + \varphi^2 \nabla_{\alpha'}\alpha', \\
  \mathbf e_3 
    &= \varphi \nabla_{\alpha'}\mathbf e_2
     = (\varphi \varphi'^2 + \varphi^2 \varphi'') \alpha' 
       + 3\varphi^2 \varphi' \nabla_{\alpha'}\alpha'
       + \varphi^3 \nabla^2_{\alpha'}{\alpha'}. \nonumber
\end{align}

Having (\ref{e1e2e3}), we find 
\begin{align*}
  \varOmega(\mathbf e_1, \mathbf e_2, \mathbf e_3) 
	= \varphi^6 \varOmega(\alpha', \nabla_{\alpha'}\alpha', \nabla^2_{\alpha'}\alpha')
	= \varepsilon. 
\end{align*}
Hence, it follows that
\begin{align*}
  \varOmega(\mathbf e_1, \mathbf e_2, \nabla_{\mathbf e_1}\mathbf e_3) = 0.
\end{align*}
Therefore, there are two functions $\varkappa_1$ and $\varkappa_2$ on $I$ such that
\begin{align}
\label{defvarkap}
  \nabla_{\mathbf e_1}\mathbf e_3 = \varkappa_1 \mathbf e_1 + \varkappa_2 \mathbf e_2.
\end{align}
The functions $\varkappa_1$, $\varkappa_2$ are called the first and the second equi-affine curvatures of the curve $\alpha$. They can be described as it follows 
\begin{align}
\label{kp1}
  \varkappa_1 
	&= \varepsilon \varOmega(\nabla_{\mathtt e_1}\mathtt e_3, \mathtt e_2, \mathtt e_3) 
	 = \varepsilon \varphi \varOmega(\nabla_{\alpha'}\mathtt e_3, \mathtt e_2, \mathtt e_3), \\[+3pt]
\label{kp2}
  \varkappa_2 
	&= \varepsilon \varOmega(\mathtt e_1, \nabla_{\mathtt e_1}\mathtt e_3, \mathtt e_3)
	 = \varepsilon \varphi \varOmega(\mathtt e_1, \nabla_{\alpha'}\mathtt e_3, \mathtt e_3).
\end{align}

\begin{theorem}
The equi-affine curvatures of a non-degenerate curve in a 3-dimensional equi-affine manifold can be expressed in the following way
\begin{align}
\label{kap1}
\varkappa_1 =\ 
	&\varphi^2 \varphi''' - 6 \varphi'^3
	        + \varepsilon \varphi^8 \varphi' 
	        \varOmega\big(\alpha', \nabla^2_{\alpha'}\alpha', 
		                    \nabla^3_{\alpha'}\alpha'\big) 
					+ \varepsilon \varphi^9
	        \varOmega\big(\nabla_{\alpha'}\alpha', 
	                       \nabla^2_{\alpha'}\alpha', 
		                     \nabla^3_{\alpha'}\alpha'\big), \\[+3pt]
\label{kap2} 
  \varkappa_2 =\ 
	&4 \varphi \varphi'' + 7 \varphi'^2 
	 - \varepsilon \varphi^8 
		 \varOmega\big(\alpha', \nabla^2_{\alpha'}\alpha', 
		               \nabla^3_{\alpha'}\alpha'\big).
\end{align}
\end{theorem}

\begin{proof}
The covariant differentiating of the third equation of (\ref{e1e2e3}) along the curve gives
\begin{align*}
  \nabla_{\alpha'}{\mathtt e}_3
    &= \big(\varphi'^3 
            + 4 \varphi \varphi' \varphi'' 
            + \varphi^2 \varphi'''\big) \alpha' 
       + \big(7 \varphi \varphi'^2 + 4 \varphi^2 \varphi''\big) 
              \nabla_{\alpha'}\alpha' \\
  &\quad \ + 6 \varphi^2 \varphi' \nabla^2_{\alpha'}\alpha' 
     + \varphi^3 \nabla^3_{\alpha'}\alpha'. 
\end{align*} 
Using the last formula and (\ref{e1e2e3}) into (\ref{kp1}) and (\ref{kp2}), after certain long computations, we find 
\begin{align}
\label{vkap1}
  \varkappa_1 =\ 
	&\varepsilon \varphi^6 
	 \big(\null- 6 \varphi \varphi' \varphi'' 
	      + 6 \varphi'^3 + \varphi^2 \varphi'''\big) 
	 \varOmega\big(\alpha', \nabla_{\alpha'}\alpha', 
		             \nabla^2_{\alpha'}\alpha'\big) \\
	&\null+ \varepsilon \varphi^7 \big(2 \varphi'^2 - \varphi \varphi''\big) 
	 \varOmega\big(\alpha', \nabla_{\alpha'}\alpha', 
		             \nabla^3_{\alpha'}\alpha'\big) \nonumber\\
	&\null+ \varepsilon \varphi^8 \varphi' 
	 \varOmega\big(\alpha', \nabla^2_{\alpha'}\alpha', 
		             \nabla^3_{\alpha'}\alpha'\big) 
	 + \varepsilon \varphi^9
	 \varOmega\big(\nabla_{\alpha'}\alpha', 
	               \nabla^2_{\alpha'}\alpha', 
		             \nabla^3_{\alpha'}\alpha'\big), \nonumber\\[+3pt]
\label{vkap2}
  \varkappa_2 =\ 
	&\varepsilon \varphi^6 
	   \big(\null- 11 \varphi'^2 + 4 \varphi \varphi''\big) 
		 \varOmega\big(\alpha', \nabla_{\alpha'}\alpha', 
		               \nabla^2_{\alpha'}\alpha'\big) \\
	&\null- 3 \varepsilon \varphi^7 \varphi' 
	   \varOmega\big(\alpha', \nabla_{\alpha'}\alpha', 
		               \nabla^3_{\alpha'}\alpha'\big) 
		 - \varepsilon \varphi^8 
		 \varOmega\big(\alpha', \nabla^2_{\alpha'}\alpha', 
		               \nabla^3_{\alpha'}\alpha'\big). \nonumber
\end{align}
When we apply the formulas
\begin{align*}
  \varOmega\big(\alpha', \nabla_{\alpha'}\alpha', \nabla^2_{\alpha'}\alpha'\big) 
	 &= \varepsilon \varphi^{-6}, \\[+3pt]
  \varOmega\big(\alpha', \nabla_{\alpha'}\alpha', \nabla^3_{\alpha'}\alpha'\big) 
	 &= -6 \varepsilon \varphi^{-7} \varphi', 
\end{align*}
the equalities (\ref{vkap1}) and (\ref{vkap2}) turn into (\ref{kap1}) and (\ref{kap2}), completing the proof.
\end{proof}

\begin{remark}
For curves in the affine 3-dimensional space, some of the above formulas occur e.g. in \cite{Bla,Gug,Su}.
\end{remark}

\section{3-dimensional pseudo-Riemannian manifolds}

Let $(M^3,g)$ be a $3$-dimensional, connected, oriented pseudo-Riemannian manifold. 

By $\varOmega$, we denote the natural volume form generated by the pseudo-Riemannian metric $g$. We assume that $\varOmega$ is compatible with the orientation of the manifold $M^3$. Thus, in the local coordinates $(x^1,x^2,x^3)$ of any chart belonging to the oriented atlas, $\varOmega$ is given by
\begin{equation*}
  \varOmega = 3! \sqrt{|G|}\, dx^1 \wedge dx^2 \wedge dx^3, 
\end{equation*}
where $G=\Det[g_{ij}]$ and $g_{ij}=g(\partial/\partial x^i,\partial/\partial x^j)$ are the local components of the metric $g$. As the skew-symmetric $(0,3)$-tensor field, the form $\varOmega$ has local components 
\begin{equation*}
  \varOmega_{ijk} = \sqrt{|G|}\, \varepsilon_{ijk}, 
\end{equation*}
where $\varepsilon_{ijk}$ are the Levi-Civita symbols (that is, $\varepsilon_{ijk}$ equals the sign ($=\pm1$) of the permutation $(i,j,k)$, or 0 if the triple $(i,j,k)$ is not a permutation of $(1,2,3)$). In the sequel, $\omega$ denotes the sign of the determinant $G$. Thus, $\omega=+1$ if the signature of the metric $g$ is $(+++)$ or $(+--)$, and $\omega=-1$ if the signature of the metric $g$ is $(++-)$ or $(---)$.

For a triple of tangent vectors $u,v,w\in T_pM^3$, $p\in M^3$, it holds: 

(a) $\varOmega(u,v,w)\neq 0$ if and only if $u,v,w$ are linearly independent; 

(b) $\varOmega(u,v,w)>0$ if and only if the frame $(u,v,w)$ is positively oriented, and $\varOmega(u,v,w)<0$ if and only if the frame $(u,v,w)$ is negatively oriented; 

(c) $\varOmega(u,v,w)=\pm1$ when $u,v,w$ are orthonormal. 

As it is well-known, the volume form $\varOmega$ enables to define the vector cross product of vector fields on $M^3$ in the following way: if $X,Y\in\mathfrak X(M^3)$ ($\mathfrak X(M^3)$ is the Lie algebra of smooth vector fields on $M^3$), then the vector  cross product $X\times Y\in\mathfrak X(M^3)$ is defined uniquely by demanding that the following condition is fulfilled
\begin{align*}
  g(X\times Y,Z) = \varOmega(X,Y,Z) \ \text{ for any } \,Z\in\mathfrak X(M^3).
\end{align*}
When using local coordinates, the vector cross product is given by 
\begin{align*}
   X\times Y &= \sum_{a,b=1}^3 \varOmega\Big(X,Y,\frac{\partial}{\partial x^a}\Big)
			            g^{ab} \frac{\partial}{\partial x^b} \\
						 &= \sqrt{|G|} \sum_{i,j,a,b=1}^3 X^i Y^j \varepsilon_{ija}
			            g^{ab} \frac{\partial}{\partial x^b} \nonumber
\end{align*}
Denote by $\mathcal F(M^3)$ the space of smooth functions on the manifold $M^3$. The vector cross product is an $\mathcal F(M^3)$-bilinear mapping from $\mathfrak X(M^3) \times \mathfrak X(M^3)$ into $\mathfrak X(M^3)$, which has the following properties 
\begin{align*}
  & X \times Y = - Y \times X, \\
  & g(X \times Y,X)=0, \\
  & (X \times Y) \times Z = \omega (\null - g(Y,Z)X + g(X,Z)Y), \\
  & (X \times Y) \times Z + (Y \times Z) \times X + (Z \times X) \times Y =0, \\
  & g(X\times Y,X\times Y) = \omega (g(X,X) g(Y,Y) - (g(X,Y))^2), \\
  & g(X\times Y,Z\times W) = \omega (g(X,Z) g(Y,W) - g(X,W) g(Y,Z)), \\
  & (X \times Y) \times (Z \times W) = \omega (\varOmega(X,Y,W)Z - \varOmega(X,Y,Z)W) 
\end{align*}
for any $X,Y,Z,W\in\mathfrak X(M^3)$. The above properties of the vector cross product are used in the next sections.

Moreover, it can be also checked that 
\begin{align*}
  \varOmega(U_1,U_2,U_3) \varOmega(V_1,V_2,V_3)
	= \omega \Det([g(U_i,V_j)]_{1\leqslant i,j\leqslant3})
\end{align*}
for any $U_1,U_2,U_3,V_1,V_2,V_3\in\mathfrak X(M^3)$.

\section{Frenet curves in 3-dimensional pseudo-Riemannian manifolds}

In this section, we state the standard procedure defining the Frenet frame and curvatures of a curve in a 3-dimensional pseudo-Riemannian manifold.

Let $M^3$ be a $3$-dimensional oriented pseudo-Riemannian manifold. Let $\alpha\colon I\to M^3$ be a non-null curve, $I$ being an open interval. Thus, $g(\alpha',\alpha')\neq0$ at any point of $I$. We assume that the curve is arc length parametrized and suppose $\varepsilon_1 = g(\alpha',\alpha') (= \pm1)$. 

Then, $\mathbf T = \alpha'$ is the unit tangent vector field along $\alpha$, and  $g(\mathbf T,\mathbf T)=\varepsilon_1$. 

We consider only non-degenerate curves, and without loss of generality, we assume that $g(\nabla_{\mathbf T}\mathbf T,\nabla_{\mathbf T}\mathbf T) \neq 0$ at every point of $I$, and let 
\begin{align*}
  \varepsilon_2 = \sign(g(\nabla_{\mathbf T}\mathbf T,\nabla_{\mathbf T}\mathbf T)) (= \pm1). 
\end{align*}

Choose the principal unit normal vector field $\mathbf N$ and the positive function $\kappa$ such that 
\begin{equation}
\label{fren1}
  \nabla_{\mathbf T}\mathbf T = \varepsilon_2\kappa \mathbf N.
\end{equation}
Thus, $g(\mathbf T,\mathbf N) = 0$ and $g(\mathbf N,\mathbf N) = \varepsilon_2$. Next, define the binormal vector field $\mathbf B$ by $\mathbf B = \varepsilon_3 \mathbf T \times \mathbf N$, where $\varepsilon_3 = \omega \varepsilon_1 \varepsilon_2$. Note that then, we have $g(\mathbf T,\mathbf B) = g(\mathbf N,\mathbf B) = 0$, $g(\mathbf B,\mathbf B) = \varepsilon_3$ and 
$\mathbf T \times \mathbf N = \varepsilon_3 \mathbf B$, 
$\mathbf N \times \mathbf B = \varepsilon_1 \mathbf T$, 
$\mathbf B \times \mathbf T = \varepsilon_2 \mathbf N$, 
and the orthonormal frame $(\mathbf T, \mathbf N, \mathbf B)$ is positively oriented since $\varOmega(\mathbf T, \mathbf N, \mathbf B) = 1$. 

Next, one checks that it holds 
\begin{align}
\label{fren2}
  \nabla_{\mathbf T} \mathbf N 
  = \null- \varepsilon_1 \kappa \mathbf T + \varepsilon_3 \tau \mathbf B
\end{align} 
for a certain function $\tau$. Finally, it also holds 
\begin{align}
\label{fren3}
  \nabla_{\mathbf T} \mathbf B = \null- \varepsilon_2 \tau \mathbf N. 
\end{align} 

The triple of the orthonormal vector fields $(\mathbf T, \mathbf N, \mathbf B)$ is the Frenet frame, $\kappa$ is the curvature and $\tau$ is the torsion of the curve $\alpha$. The equations (\ref{fren1}) - (\ref{fren3}) are the Frenet equations of this curve. 

Below, we write formulas useful in the description of the Frenet frame, the curvature and torsion of an arc length parametrized curve

\begin{align*}
  \mathbf N 
  &= \varepsilon_3 
     \frac{(\alpha' \times \nabla_{\alpha'} \alpha') 
               \times \alpha'}
          {\|\alpha' \times \nabla_{\alpha'} \alpha'\|} 
   = \omega \varepsilon_3 
     \frac{g(\alpha',\alpha') \nabla_{\alpha'}\alpha'
           - g(\alpha', \nabla_{\alpha'}\alpha') 
           \alpha'}
          {\|\alpha' \times \nabla_{\alpha'} \alpha'\|}, \\[+3pt]
  \mathbf B 
	&= \omega \varepsilon_1 
     \frac{\alpha' \times \nabla_{\alpha'} \alpha'}
          {\|\alpha' \times \nabla_{\alpha'} \alpha'\|},\\[+3pt]
  \kappa 
	&= \|\alpha' \times \nabla_{\alpha'} \alpha'\|, \\[+3pt]
  \tau 
	&= \varepsilon_3 
     \frac{\varOmega\left(\alpha', \nabla_{\alpha'} \alpha', 
           \nabla^2_{\alpha'} \alpha'\right)}
 	 	 	 	 	{\|\alpha' \times \nabla_{\alpha'} \alpha'\|^2}. 
\end{align*}

\section{Equi-affine curvatures of Frenet curves in 3-dimensional pseudo-Riemannian manifolds}

Let $(M^3,g)$ be a $3$-dimensional, connected, oriented pseudo-Riemannian manifold. We will also treat this manifold as the equi-affine manifold $(M^3,\nabla,\varOmega)$ with $\nabla$ and $\varOmega$ being the Levi-Civita connection and the natural volume element related to the pseudo-Riemannian metric $g$.

Let $\alpha\colon I\to M^3$ be a Frenet curve parametrized by the arc length. Assume also that this curve is non-degenerate and suppose
\begin{align}
\label{vareps}
  \varepsilon 
	  = \sign\left(\varOmega\left(\alpha', \nabla_{\alpha'}\alpha', \nabla^2_{\alpha'}\alpha'\right)\right).
\end{align}

In this section, we are going to express the Cartan moving frame $(\mathbf e_1, \mathbf e_2, \mathbf e_3)$ with the help of the Frenet frame $(\mathbf T, \mathbf N, \mathbf B)$, and the equi-affine curvatures $\varkappa_1$, $\varkappa_2$ with the help of	 the Frenet curvature $\kappa$ and torsion $\tau$ of the curve $\alpha$.

At first, we give some consequences of the the Frenet equations, which will be used below. Namely, having $\alpha'=\mathbf T$ and the equations (\ref{fren1}) - (\ref{fren3}), we compute 
\begin{align*}
  \nabla_{\alpha'} \alpha' 
	  &=\nabla_{\mathbf T}\mathbf T
		 = \varepsilon_2 \kappa \mathbf N, \\[+3pt]
  \nabla^2_{\alpha'} \alpha'
    &= \nabla_{\alpha'} \nabla_{\alpha'}\alpha' 
		 = -\, \varepsilon_1 \varepsilon_2 \kappa^2 \mathbf T 
		   + \varepsilon_2 \kappa' \mathbf N 
			 + \varepsilon_2 \varepsilon_3 \kappa \tau \mathbf B.
\end{align*}
Next, since $\varOmega(\mathbf T,\mathbf N,\mathbf B)=1$, using the formulas, we find
\begin{align}
\label{omega-1}
  \varOmega\left(\alpha', \nabla_{\alpha'} \alpha', \nabla^2_{\alpha'} \alpha'\right) 
		 &= \varepsilon_3 \kappa^2 \tau.
\end{align}
From (\ref{vareps}) and (\ref{omega-1}), it follows that $\tau$ is non-zero at every point and $\varepsilon = \varepsilon_3 \sign \tau$. Therefore, $\sign \tau = \varepsilon \varepsilon_3$. Moreover, by applying (\ref{omega-1}) into (\ref{mu}), we have 
\begin{align}
\label{phi}
  \varphi = \kappa^{-1/3} |\tau|^{-1/6}.
\end{align}

\begin{theorem}
Let $\alpha$ be a non-null, non-degenerate and arc length parametrized curve in a 3-dimensional oriented pseudo-Riemannian manifold. Then, the Cartan frame $(\mathbf e_1, \mathbf e_2, \mathbf e_3)$ of $\alpha$ can be expressed with the help of the Frenet frame $(\mathbf T, \mathbf N, \mathbf B)$ as it follows
\begin{align}
\label{e1-riem}
  \mathbf e_1 & 
	   = \kappa^{-1/3} |\tau|^{-1/6} \mathbf T, \\[+3pt]
\label{e2-riem}
	\mathbf e_2 & 
	   = \null- 6^{-1} \kappa^{-5/3} \tau^{-1} |\tau|^{-1/3} 
		   \left(2\tau \kappa' + \kappa \tau'\right) \mathbf T 
			 + \varepsilon_2 \kappa^{1/3} |\tau|^{-1/3} \mathbf N, \\[+3pt]
\label{e3-riem}
	\mathbf e_3 &= 18^{-1} \kappa^{-3} \tau^{-2} |\tau|^{-1/2}\ \cdot \\[+3pt]
	            &\quad \cdot \left(\null- 18 \varepsilon_1 \varepsilon_2 \kappa^4 \tau^2 
								  + 10 \tau^2 \kappa'^2 + 4 \kappa \tau \kappa' \tau' 
									+ 4 \kappa ^2 \tau'^2 - 6 \kappa \tau^2 \kappa'' 
									- 3 \kappa^2 \tau \tau''\right) \mathbf T \nonumber \\[+3pt]
							& \quad- 2^{-1} \varepsilon_2 \tau^{-1} |\tau|^{-1/2} \tau' \mathbf N 
							  + \varepsilon_2 \varepsilon_3 \tau |\tau|^{-1/2} \mathbf B. \nonumber
\end{align}
Moreover, the equi-affine curvatures $\varkappa_1$, $\varkappa_2$ of $\alpha$ can be expressed with the help of the Frenet curvature $\kappa$ and torsion $\tau$ as it follows
\begin{align}
\label{kap1-riem}
  \varkappa_1 = &\,216^{-1} \kappa^{-4} \tau^{-3}|\tau|^{-1/2} 
	   \big(\null- 288 \varepsilon _1 \varepsilon _2 \kappa ^4 \tau ^3 \kappa '
		          - 72 \varepsilon _2 \varepsilon _3 \kappa ^2 \tau^5 \kappa' 
							- 320 \tau^3 \kappa'^3 \\[+3pt]
						&\null+ 180 \varepsilon_1 \varepsilon_2 \kappa^5 \tau^2 \tau' 
							- 36 \varepsilon_2 \varepsilon_3 \kappa^3 \tau^4 \tau' 
							- 120 \kappa \tau^2 \kappa'^2 \tau' 
							- 42 \kappa^2 \tau \kappa' \tau'^2 \nonumber \\[+3pt]
						&\null- 85 \kappa^3 \tau'^3 
							+ 360 \kappa \tau^3 \kappa' \kappa'' 
							+ 72 \kappa^2 \tau^2 \tau' \kappa'' 
							+ 36 \kappa^2 \tau^2 \kappa' \tau'' 
							+ 126 \kappa^3 \tau \tau' \tau'' \nonumber \\[+3pt]
						&\null- 72 \kappa^2 \tau^3 \kappa''' 
							- 36 \kappa^3 \tau^2 \tau'''\big), \nonumber \\[+3pt]
\label{kap2-riem}
	\varkappa_2 = &\,36^{-1} \kappa^{-8/3} \tau^{-2} |\tau|^{-1/3}  
	            \big(\null- 36 \varepsilon_1 \varepsilon_2 \kappa^4 \tau^2  
							- 36 \varepsilon_2 \varepsilon_3  \kappa^2 \tau^4 \\[+3pt]
						 &\null + 20 \tau^2 \kappa'^2 
							 + 8 \kappa \tau \kappa' \tau' 
						   + 35 \kappa ^2 \tau'^2 
							 - 12 \kappa \tau^2 \kappa'' 
							 - 24 \kappa^2 \tau \tau''\big). \nonumber
\end{align}
\end{theorem}

\begin{proof}
At first, using (\ref{defCartfr}) and $\alpha' = \mathbf T$, we obtain $\mathbf e_1 = \varphi \alpha' = \varphi \mathbf T$, which by (\ref{phi}) gives (\ref{e1-riem}). Next, having (\ref{e1-riem}), we compute
\begin{align*}
  \mathbf e_2 
	&= \nabla_{\mathbf e_1}\mathbf e_1 \\[+3pt]
	&= \kappa^{-1/3} |\tau|^{-1/6} 
		 \nabla_{\mathbf T}\left(\kappa^{-1/3} |\tau|^{-1/6} \mathbf T\right) \\[+3pt]
	&= \null- 6^{-1} \kappa^{-5/3} \tau^{-1} |\tau|^{-1/3} 
		 \left(2\tau \kappa' + \kappa \tau'\right) \mathbf T 
			+ \kappa^{-2/3} |\tau|^{-1/3} \nabla_{\mathbf T}\mathbf T,
\end{align*}
which together with (\ref{fren1}) leads to (\ref{e2-riem}). In a similar manner, but having (\ref{e1-riem}) and (\ref{e2-riem}), we find
\begin{align*}
  \mathbf e_3 
	&= \nabla_{\mathbf e_1}\mathbf e_2 \\[+3pt]
  &= \kappa^{-1/3} |\tau|^{-1/6} 
		 \nabla_{\mathbf T}\left(\null- 6^{-1} \kappa^{-5/3} \tau^{-1} |\tau|^{-1/3} 
		   \left(2\tau \kappa' + \kappa \tau'\right) \mathbf T 
			 + \varepsilon_2 \kappa^{1/3} |\tau|^{-1/3} \mathbf N\right) \\[+3pt]
	&= 18^{-1} \kappa^{-3} \tau^{-2} |\tau|^{-1/2} 
	   \left(10 \tau^2 \kappa'^2 + 4 \kappa \tau \kappa' \tau' + 4 \kappa^2 \tau'^2 
		       - 6 \kappa \tau^2 \kappa'' - 3 \kappa^2 \tau \tau''\right) \mathbf T \\[+3pt]
	&\quad- 6^{-1} \kappa^{-2} \tau^{-1} |\tau|^{-1/2} 
	        \left(2 \tau \kappa' + \kappa \tau'\right) \nabla_{\mathbf T}\mathbf T \\[+3pt]
	&\quad+ 3^{-1} \varepsilon_2 \kappa^{-1} \tau^{-1} |\tau|^{-1/2} 
	        \left(\tau \kappa' - \kappa \tau'\right) \mathbf N 
					+ \varepsilon_2 |\tau|^{-1/2} \nabla_{\mathbf T}\mathbf N.
\end{align*}
Hence, applying (\ref{fren1}) and (\ref{fren2}), we obtain (\ref{e3-riem}). 

To prove (\ref{kap1-riem}) and (\ref{kap2-riem}) we express the covariant derivative $\nabla_{\mathbf e_1}\mathbf e_3$ with the help of the Frenet vector fields $\mathbf T$, $\mathbf N$, $\mathbf B$. To do it, having (\ref{e3-riem}), we compute 
\begin{align*}
  \nabla_{\mathbf e_1}\mathbf e_3 
    &= \kappa^{-1/3} |\tau|^{-1/6} \nabla_{\mathbf T}\mathbf e_3 \\[+3pt]
		&= 36^{-1} \kappa^{-13/3} \tau^{-11/3} 
		   \big(\null- 36 \varepsilon_1 \varepsilon_2 \kappa^4 \tau^3 \kappa' 
	               - 60 \tau^3 \kappa'^3 
	               + 18 \varepsilon_1 \varepsilon_2 \kappa^5 \tau^2 \tau' \\[+3pt]
	  &\quad- 26 \kappa \tau^2 \kappa'^2 \tau' 
	               - 20 \kappa^2 \tau \kappa' \tau'^2 
	               - 20 \kappa^3 \tau'^3 
	               + 64 \kappa \tau^3 \kappa' \kappa'' 
                 + 14 \kappa^2 \tau^2 \tau' \kappa'' \\[+3pt]
    &\quad+ 14 \kappa^2 \tau^2 \kappa' \tau'' 
	               + 25 \kappa^3 \tau \tau' \tau'' 
	               - 12 \kappa^2 \tau^3 \kappa'''
	               - 6 \kappa^3 \tau^2 \tau'''\big) \mathbf T \\[+3pt]
	  &\quad+ 18^{-1} \kappa^{-10/3} \tau^{-8/3} \\[+3pt]
		&\quad \left(\null- 18 \varepsilon_1 \varepsilon_2 \kappa^4 \tau^2 
								 + 10 \tau^2 \kappa'^2 + 4 \kappa \tau \kappa' \tau' 
								 + 4 \kappa ^2 \tau'^2 - 6 \kappa \tau^2 \kappa'' 
								 - 3 \kappa^2 \tau \tau''\right) \nabla_{\mathbf T}{\mathbf T} \\[+3pt]
		&\quad+ 4^{-1} \varepsilon_2 \kappa^{-1/3} \tau^{-8/3} 
		     \left(3 \tau'^2 - 2 \tau \tau''\right) \mathbf N 
		     - 2^{-1} \varepsilon_2 \kappa^{-1/3} \tau^{-5/3} \tau' 
		              \nabla_{\mathbf T}\mathbf N \\[+3pt]
		&\quad+ 2^{-1} \varepsilon_2 \varepsilon_3 \kappa^{-1/3} \tau^{-2/3} \tau' \mathbf B 
		     + \varepsilon_2 \varepsilon_3 \kappa^{-1/3} \tau^{1/3} 
				           \nabla_{\mathbf T}\mathbf B.
\end{align*}
Hence, applying the Frenet equations (\ref{fren1}) - (\ref{fren3}), we obtain
\begin{align*}
  \nabla_{\mathbf e_1}\mathbf e_3 
	&= 36^{-1} \kappa^{-13/3} \tau^{-11/3} 
	   \big(\null- 36 \varepsilon_1 \varepsilon_2 \kappa^4 \tau^3 \kappa'
	             - 60 \tau ^3 \kappa'^3
							 + 36 \varepsilon_1 \varepsilon_2 \kappa^5 \tau^2 \tau' \\[+3pt]
	&\quad- 26 \kappa \tau^2 \kappa'^2 \tau '
	      - 20 \kappa^2 \tau \kappa' \tau'^2
	      - 20 \kappa^3 \tau'^3
	      + 64 \kappa \tau^3 \kappa' \kappa''
	      + 14 \kappa^2 \tau^2 \tau' \kappa'' \nonumber\\[+3pt]
	&\quad+ 14 \kappa^2 \tau^2 \kappa' \tau''
	      + 25 \kappa^3 \tau \tau' \tau''
	      - 12 \kappa^2 \tau ^3 \kappa'''
	      - 6 \kappa^3 \tau^2 \tau''' \big) \mathbf T \nonumber\\[+3pt]
	&\quad+ 36^{-1} \kappa^{-13/3} \tau^{-11/3} 
	   \big(\null- 36 \varepsilon_1 \kappa^6 \tau^3 
	      - 36 \varepsilon_3 \kappa^4 \tau^5 
	      + 20 \varepsilon_2 \kappa^2 \tau^3 \kappa'^2 \nonumber\\[+3pt]
	&\quad+ 8 \varepsilon_2 \kappa^3 \tau^2 \kappa' \tau'
	      + 35 \varepsilon_2 \kappa^4 \tau \tau'^2
		    - 12 \varepsilon_2 \kappa^3 \tau^3 \kappa ''
	      - 24 \varepsilon_2 \kappa^4 \tau^2 \tau''\big) \mathbf N. \nonumber
\end{align*}
On the other hand, by (\ref{defvarkap}), (\ref{e1-riem}) and (\ref{e2-riem}) lead to
\begin{align*}
  \nabla_{\mathbf e_1}\mathbf e_3 
	&= \varkappa_1 \mathbf e_1 + \varkappa_2 \mathbf e_2 \\[+3pt]
	&= \left(\varkappa_1 \kappa^{-1/3} |\tau|^{-1/6} 
	         - 6^{-1} \varkappa_2 \kappa^{-5/3} \tau^{-1} |\tau|^{-1/3} 
		       \left(2\tau \kappa' + \kappa \tau'\right)\right) \mathbf T \\[+3pt]
	&\quad\null+ \varkappa_2 \varepsilon_2 \kappa^{1/3} |\tau|^{-1/3} \mathbf N
\end{align*}
Comparing the last two expressions for $\nabla_{\mathbf e_1}\mathbf e_3$, we obtain the following system of linear equations with respect to $\varkappa_1$ and $\varkappa_2$
\begin{align*}
  &\varkappa_1 \kappa^{-1/3} |\tau|^{-1/6} 
	         - 6^{-1} \varkappa_2 \kappa^{-5/3} \tau^{-1} |\tau|^{-1/3} 
		       \left(2\tau \kappa' + \kappa \tau'\right) \\[+3pt]
	&= 36^{-1} \kappa^{-13/3} \tau^{-11/3} 
	   \big(\null- 36 \varepsilon_1 \varepsilon_2 \kappa^4 \tau^3 \kappa'
	             - 60 \tau ^3 \kappa'^3
							 + 36 \varepsilon_1 \varepsilon_2 \kappa^5 \tau^2 \tau' \\[+3pt]
	&\quad- 26 \kappa \tau^2 \kappa'^2 \tau '
	      - 20 \kappa^2 \tau \kappa' \tau'^2
	      - 20 \kappa^3 \tau'^3
	      + 64 \kappa \tau^3 \kappa' \kappa''
	      + 14 \kappa^2 \tau^2 \tau' \kappa'' \nonumber\\[+3pt]
	&\quad+ 14 \kappa^2 \tau^2 \kappa' \tau''
	      + 25 \kappa^3 \tau \tau' \tau''
	      - 12 \kappa^2 \tau ^3 \kappa'''
	      - 6 \kappa^3 \tau^2 \tau''' \big), \\[+3pt]
	&\varkappa_2 \varepsilon_2 \kappa^{1/3} |\tau|^{-1/3} \\[+3pt]
	&= 36^{-1} \kappa^{-13/3} \tau^{-11/3} 
	   \big(\null- 36 \varepsilon_1 \kappa^6 \tau^3 
	      - 36 \varepsilon_3 \kappa^4 \tau^5 
	      + 20 \varepsilon_2 \kappa^2 \tau^3 \kappa'^2 \nonumber\\[+3pt]
	&\quad+ 8 \varepsilon_2 \kappa^3 \tau^2 \kappa' \tau'
	      + 35 \varepsilon_2 \kappa^4 \tau \tau'^2
		    - 12 \varepsilon_2 \kappa^3 \tau^3 \kappa ''
	      - 24 \varepsilon_2 \kappa^4 \tau^2 \tau''\big).
\end{align*}
Solving the above system, we obtain (\ref{kap1-riem}) and (\ref{kap2-riem}).
\end{proof}

The below corollaries are consequences of the above theorem. We use them in the next section.

\begin{corollary}
\label{coro1}
Let $\alpha$ be a non-null, non-degenerate and arc length parametrized curve in a 3-dimensional oriented pseudo-Riemannian manifold. If the Frenet curvature $\kappa$ and torsion $\tau$ are constant, then the equiaffine curvatures are also constant and 
\begin{align*}
  \varkappa_1 = 0, \quad 
	\varkappa_2 = -\dfrac{\varepsilon_2 (\varepsilon_1 \kappa^2 + \varepsilon_3 \tau^2)}
	                     {\kappa^{2/3} |\tau|^{1/3}}.
\end{align*}
\end{corollary}

\begin{corollary}
\label{coro2}
Let $\alpha$ be a non-null, non-degenerate and arc length parametrized curve in a 3-dimensional oriented pseudo-Riemannian manifold. If the Frenet curvature and torsion of $\alpha$ are given by 
\begin{align*}
  \kappa(t) = \dfrac{A}{t}, \quad \tau(t) = \dfrac{B}{t}, \quad t > 0,
\end{align*}
where $A$ and $B$ are non-zero constants, then the equi-affine curvatures of $\alpha$ are 
\begin{align*}
  \varkappa_1 = \dfrac{1 + 4 \varepsilon_2(\varepsilon_1 A^2 + \varepsilon_3 B^2)}
	                     {8 A \sqrt{|B|\, t^3}}, \quad 
	\varkappa_2 = - \dfrac{1 + 4 \varepsilon_2(\varepsilon_1 A^2 + \varepsilon_3 B^2)}
	                      {4 \sqrt[3]{A^2 |B|}\, t}.
\end{align*}
\end{corollary}

\section{Examples of Frenet curves with constant equi-affine curvatures}

As it is stated in Corollary \ref{coro1}, the constancy of the curvature and torsion of a Frenet curve always implies the constancy of the equi-affine curvatures. Frenet curves with constant curvature and torsion are usually called helices, and we can add that there are very many examples of such curves. 

As it follows from the below discussion and examples, the statement which is converse to that given in Corollary \ref{coro1} does not hold in general.

Curves of constant equi-affine curvatures in the tree dimensional affine space $\mathbb A^3 = (\mathbb R^3, \text{D}, \Det)$ ($\text{D}$ being the standard flat connection in the Cartesian space $\mathbb R^3$) are already classified (\cite{Bla,Gug}). It is obvious that the most of these curves, when treated as curves in the Euclidean space $\mathbb E^3$ or the Minkowski space $\mathbb E^3_1$, do not have constant Frenet curvature and torsion. 

Below examples show that in certain non-flat Lorentzian manifolds, there exist non-degenerate curves with constant (precisely, zero) equi-affine curvatures and whose Frenet curvature and torsion are non-constant. 

\medskip
\begin{example}{1}
Let us define a Lorentzian metric $g$ in $\mathbb R^3$ by 
\begin{align*}
  g  = e^{2z}(dx^2 + dy^2 - dz^2),
\end{align*}
where $(x^1=x,x^2=y,x^3=z)$ are the Cartesian coordinates in $\mathbb R^3$. The Levi-Civita connection $\nabla$ is given by
\begin{eqnarray*}
  &\nabla_{\partial_x}\partial_x 
	= \nabla_{\partial_y}\partial_y 
	= \nabla_{\partial_z}\partial_z 
	= \partial_z, \quad
	\nabla_{\partial_x}\partial_y 
	= \nabla_{\partial_y}\partial_x 
	= 0, \\
	&\nabla_{\partial_x}\partial_z 
	= \nabla_{\partial_z}\partial_x 
	= \partial_x, \quad
	\nabla_{\partial_y}\partial_z 
	=\nabla_{\partial_z}\partial_y 
	= \partial_y, &
\end{eqnarray*}
where $\partial_x = \partial/\partial x$, $\partial_y = \partial/\partial y$, $\partial_z = \partial/\partial z$. This metric is Lorentzian, non-flat, conformally flat with non-constant positive scalar curvature, but we omit the details. The volume element is 
\begin{align*}
  \varOmega = 3! e^{3z} dx \wedge dy \wedge dz.
\end{align*}

In $(\mathbb R^3,g)$ , consider the curve 
\begin{align*}
  \alpha(t) = \left(\dfrac{\sqrt{a^2+\lambda}}{b} \cos\psi(t), 
	                  \dfrac{\sqrt{a^2+\lambda}}{b} \sin\psi(t), 
										\ln(a t)\right),
	\ t \in (0,\infty),
\end{align*}
where $\psi(t) = (b/a) \ln(a t)$, $\lambda=\pm1$, and $a,b$ are constants such that $a > 0$, $a^2+\lambda>0$, $b \neq 0$, $b^2 - \lambda >0$. By direct computations, we find 
\begin{align*}
  \alpha'(t) &= \null- \dfrac{\sqrt{a^2+\lambda}}{at} \sin \psi(t)\, 
	                     {\partial_x}\big|_{\alpha(t)} 
	              + \dfrac{\sqrt{a^2+\lambda}}{at} \cos \psi(t)\, 
								       {\partial_y}\big|_{\alpha(t)}
								+ \dfrac{1}{t} \,{\partial_z}\big|_{\alpha(t)}, \\[+3pt]
  \big(\nabla_{\alpha'}\alpha'\big)(t) &= 
	       \null- \dfrac{\sqrt{a^2+\lambda}}{a^2 t^2} 
				        \left(b \cos \psi(t) + a \sin \psi(t)\right) 
				        {\partial_x}\big|_{\alpha(t)} \\[+3pt]
	       &\qquad+ \dfrac{\sqrt{a^2+\lambda}}{a^2 t^2} 
				        \left(a \cos \psi(t) - b \sin \psi(t)\right) 
							  {\partial_y}\big|_{\alpha(t)} 
				  + \dfrac{a^2+\lambda}{a^2 t^2} \,{\partial_z}\big|_{\alpha(t)}, \\[+3pt]
	\big(\nabla^2_{\alpha'}\alpha'\big)(t) &= 
	        \dfrac{(b^2-\lambda) \sqrt{a^2+\lambda}}{a^3 t^3} 
				  \left(\sin \psi(t)\, {\partial_x}\big|_{\alpha(t)}
	              - \cos \psi(t)\, {\partial_y}\big|_{\alpha(t)}\right).
\end{align*}
Therefore, it can be derived that 
\begin{eqnarray*}
  & g(\alpha',\alpha') = \lambda, \quad 
	  g(\nabla_{\alpha'}\alpha',\nabla_{\alpha'}\alpha') 
		  = \dfrac{(a^2+\lambda)(b^2-\lambda)}{a^2 t^2}, \\[+3pt]
	& \varOmega\left(\alpha', \nabla_{\alpha'} \alpha', \nabla^2_{\alpha'} \alpha'\right) = 
	  \dfrac{b(a^2+\lambda)(b^2-\lambda)}{a^2 t^3}.
\end{eqnarray*}
Moreover, we find 
\begin{align*}
	\alpha'(t) \times \big(\nabla_{\alpha'}\alpha'\big)(t) 
	 &= \dfrac{\sqrt{a^2+\lambda}}{a^2 t^2} 
			\left(\lambda \cos \psi(t) + a b \sin \psi(t)\right) 
			{\partial_x}\big|_{\alpha(t)} \\[+3pt]
	 &\null- \dfrac{\sqrt{a^2+\lambda}}{a^2 t^2} 
			\left(a b \cos \psi(t) - \lambda \sin \psi(t)\right) 
		{\partial_y}\big|_{\alpha(t)} 
	  - \dfrac{b(a^2+\lambda)}{a^2 t^2} \,{\partial_z}\big|_{\alpha(t)},
\end{align*}
and consequently, 
\begin{align*}
	 g\left(\alpha'(t) \times \left(\nabla_{\alpha'}\alpha'\right)(t), 
	   \alpha'(t) \times \big(\nabla_{\alpha'}\alpha'\big)(t)\right) 
	 = - \dfrac{\lambda (a^2 + \lambda)(b^2 - \lambda)}{a^2 t^2}. 
\end{align*}
Under our assumptions about the constants $a,b$, we deduce that $\nu(t) = 1$, $\varepsilon_1 = \lambda$, $\varepsilon_2 = 1$, $\varepsilon_3 = -\lambda$, $\varepsilon = \sign b$ and 
\begin{align*}
	\kappa(t) = \dfrac{\sqrt{(a^2 + \lambda)(b^2 -\lambda)}}{a t},\quad 
	\tau(t) = - \dfrac{\lambda b}{t}.
\end{align*}
The curves are non-degenerate, and spacelike when $\lambda = 1$ and timelike when $\lambda = -1$. By Corollary \ref{coro2}, these curves have the equi-affine curvatures equal to 
\begin{align*}
  \varkappa_1 &= \dfrac{\null- 3a^2 + 4(b^2-\lambda)}
	                    {8 a \sqrt{|b| (a^2+\lambda)(b^2-\lambda) t^3}}, \\
	\varkappa_2 &= \dfrac{3a^2 - 4(b^2-\lambda)}
	                     {4 a \sqrt[3]{ a |b| (a^2+\lambda)(b^2-\lambda)}\, t}.
\end{align*}
Therefore, $\varkappa_1 = \varkappa_2 = 0$ if and only if $3a^2 = 4(b^2-\lambda)$. One can easily observe that this condition can be realized by some constants for spacelike curves as well as for timelike curves.
\end{example}

\medskip
\begin{example}{2} 
Let us define a Lorentzian metric $g$ in $\mathbb R^3$ by 
\begin{align*}
  g  = e^{2x}(dx^2 + dy^2 - dz^2),
\end{align*}
where $(x^1=x,x^2=y,x^3=z)$ are the Cartesian coordinates in $\mathbb R^3$. The Levi-Civita connection $\nabla$ is given by
\begin{eqnarray*}
  &\nabla_{\frac{\partial}{\partial x}}\dfrac{\partial}{\partial x} 
	= - \nabla_{\frac{\partial}{\partial y}}\dfrac{\partial}{\partial y} 
	= \nabla_{\frac{\partial}{\partial z}}\dfrac{\partial}{\partial z}
	= \dfrac{\partial}{\partial x}, \quad
	\nabla_{\frac{\partial}{\partial x}}\dfrac{\partial}{\partial y} 
	= \nabla_{\frac{\partial}{\partial y}}\dfrac{\partial}{\partial x} 
	= \dfrac{\partial}{\partial y},& \\
	&\nabla_{\frac{\partial}{\partial x}}\dfrac{\partial}{\partial z} 
	= \nabla_{\frac{\partial}{\partial z}}\dfrac{\partial}{\partial x} 
	= \dfrac{\partial}{\partial z},\quad
	\ \ \text{and} \ \ \nabla_{\frac{\partial}{\partial x^i}}\dfrac{\partial}{\partial x^j} = 0
	\ \ \text{otherwise.}&
\end{eqnarray*}
This metric $g$ is Lorentzian, non-flat, conformally flat with non-constant negative scalar curvature, but we omit the details. The volume element is of the form
\begin{align*}
  \varOmega = 3! e^{3x} dx \wedge dy \wedge dz.
\end{align*}

In $(R^3,g)$ , consider the curve 
\begin{align*}
  \alpha(t) = \left(\ln(at), \dfrac{\sqrt{1-a^2}}{b} \sinh\psi(t), 
	                           \dfrac{\sqrt{1-a^2}}{b} \cosh\psi(t)\right),
	\quad t \in (0,\infty),
\end{align*}
where $\psi(t) = (b/a) \ln(at)$, and $a,b$ are constants such that $0<a<1$, $0<b^2<1$. At first, we find 
\begin{align*}
  \alpha'(t) &= \dfrac{1}{t}\, \,{\partial_x}\big|_{\alpha(t)}
								+ \dfrac{\sqrt{1-a^2}}{at} \left(\cosh\psi(t)\,{\partial_y}\big|_{\alpha(t)}
	              + \sinh\psi(t)\,{\partial_z}\big|_{\alpha(t)}\right), \\[+3pt]
  \big(\nabla_{\alpha'}\alpha'\big)(t) &= 
				 \dfrac{a^2-1}{a^2t^2} \,{\partial_x}\big|_{\alpha(t)}
	       + \dfrac{\sqrt{1-a^2}}{a^2t^2} \left(a \cosh\psi(t) + b \sinh\psi(t)\right)
	         {\partial_y}\big|_{\alpha(t)} \\[+3pt]
	       &\qquad+ \dfrac{\sqrt{1-a^2}}{a^2t^2} \left(b \cosh\psi(t) + a \sinh\psi(t)\right) 
								  {\partial_z}\big|_{\alpha(t)}, \\[+3pt]
	\big(\nabla^2_{\alpha'}\alpha'\big)(t) &= 
	       \dfrac{(b^2-1)\sqrt{1-a^2}}{a^3t^3} 
	       \left(\cosh\psi(t) {\partial_y}\big|_{\alpha(t)} 
				        + \sinh\psi(t) {\partial_z}\big|_{\alpha(t)}\right).
\end{align*}
Therefore, it can be derived that 
\begin{eqnarray*}
  & g(\alpha',\alpha') = 1, \quad 
	  g(\nabla_{\alpha'}\alpha',\nabla_{\alpha'}\alpha') = \dfrac{(1-a^2)(1-b^2)}{a^2t^2}, &\\[+3pt]
  & \varOmega\left(\alpha', \nabla_{\alpha'} \alpha', \nabla^2_{\alpha'} \alpha'\right) = 
	  \dfrac{b(1-a^2)(1-b^2)}{a^2t^3}. &
\end{eqnarray*}
Moreover, we find
\begin{align*}
	\alpha'(t) \times \big(\nabla_{\alpha'}\alpha'\big)(t) 
	&= \dfrac{b(1-a^2)}{a^2t^2}\,{\partial_x}\big|_{\alpha(t)}
				- \dfrac{\sqrt{1-a^2}}{a^2t^2} \left(a b \cosh\psi(t) + \sinh\psi(t)\right)
	        {\partial_y}\big|_{\alpha(t)} & \\
	&\null- \dfrac{\sqrt{1-a^2}}{a^2t^2} \left(\cosh\psi(t) + a b \sinh\psi(t)\right)
				  \,{\partial_z}\big|_{\alpha(t)}, &
\end{align*}
and consequently,
\begin{align*}
  g\left(\alpha'(t) \times \big(\nabla_{\alpha'}\alpha'\big)(t),
	       \alpha'(t) \times \big(\nabla_{\alpha'}\alpha'\big)(t)\right)
	= - \dfrac{(1-a^2)(1-b^2)}{a^2t^2}.
\end{align*}

In view of the above formulas and our assumptions about the constants $a,b$, we have $\nu(t) = \varepsilon_1 = \varepsilon_2 = 1$, $\varepsilon_3 = -1$, $\varepsilon = \sign(b)$ and 
\begin{align*}
	\kappa(t) = \dfrac{\sqrt{(1-a^2)(1-b^2)}}{at},\quad \tau(t) = -\dfrac{b}{t}.
\end{align*}
The curves are spacelike and non-degenerate. By Corollary \ref{coro2}, their equi-affine curvatures are 
\begin{align*}
  \varkappa_1 &= \dfrac{4 - 3a^2 - 4b^2}{8 a \sqrt{|b|(1-a^2)(1-b^2)t^3}}, \\[+3pt]
	\varkappa_2 &= \dfrac{3a^2 + 4b^2 - 4}{8 a \sqrt[3]{a |b| (1-a^2)(1-b^2)}\,t}.
\end{align*}
Therefore, $\varkappa_1 = \varkappa_2 = 0$ if and only if $3a^2 + 4b^2 = 4$. One easily observes the existence of the values of $a,b$ realizing this condition. 
\end{example}

\section{Equi-affine curvatures of null curves in 3-dimensional Lorentzian manifolds}

Constructing the Frenet frame and the curvature functions of a null curve in an arbitrary 3-dimensional Lorentzian manifold, we follow \cite{Dug, DJ}; cf. also \cite{FGL}.

Let $(M^3,g)$ be a $3$-dimensional, connected, oriented Lorentzian manifold. As in the previous sections, we treat this manifold as the equi-affine manifold $(M^3,\nabla,\varOmega)$ with $\nabla$ and $\varOmega$ being the Levi-Civita connection and the natural volume element related to the Lorentzian metric $g$.

Let $\alpha\colon I\to M^3$, $I$ being an open interval, be a null and non-degenerate curve in $(M^3,g)$. Thus, $g(\alpha',\alpha')=0$ and the vector fields $\alpha'$, $\nabla_{\alpha'}\alpha'$, $\nabla^2_{\alpha'}\alpha'$ are linearly independent at every point of the curve. Note that since additionally $g(\nabla_{\alpha'}\alpha',\alpha') = 0$, it must be that $g(\nabla_{\alpha'}\alpha', \nabla_{\alpha'}\alpha')>0$, and consequently 
\begin{align*}
  g(\nabla^2_{\alpha'}\alpha', \alpha') 
	= -g(\nabla_{\alpha'}\alpha', \nabla_{\alpha'}\alpha')<0.
\end{align*}
As it is well-known, the curve $\alpha$ can be reparametrized in such a way that 
\begin{align*}
  g(\nabla_{\alpha'}\alpha',
		  \nabla_{\alpha'}\alpha') = 1. 
\end{align*}
Such a parametrization is called pseudo-arc length (distinguish). In the sequel, we assume the curve $\alpha$ is pseudo-arc length parametrized.

For the curve $\alpha$, define the function $\tau$ by 
\begin{align}
\label{kap-null}
  \tau = \frac12
     g(\nabla^2_{\alpha'}\alpha',
       \nabla^2_{\alpha'}\alpha').
\end{align}
Thus, for the vector fields $\alpha'$, $\nabla_{\alpha'}\alpha'$, $\nabla^2_{\alpha'}\alpha'$, we have 
\begin{equation}
\label{scprod}
\begin{array}{|c||c|c|c|}
  \hline
  g(\cdot,\cdot) & \alpha' & \nabla_{\alpha'}\alpha' 
      & \nabla^2_{\alpha'}\alpha' \\
  \hline\hline
  \alpha' & 0 & 0 & -1 \\
  \hline
  \nabla_{\alpha'}\alpha' & 0 & 1 & 0 \\
  \hline
  \nabla^2_{\alpha'}\alpha' 
      & -1 & 0 & 2\tau \\
  \hline 
\end{array}
\end{equation}

Define the three vector fields $\mathbf L$, $\mathbf N$ and $\mathbf W$ along the curve $\alpha$ by 
\begin{align}
\label{LWN}
  \mathbf L = \alpha',\quad
  \mathbf W = \nabla_{\alpha'}\alpha',\quad 
  \mathbf N = \null -\nabla^2_{\alpha'}
                      \alpha' - \tau\alpha',
\end{align}
$\tau$ being the function defined by (\ref{kap-null}). In view of (\ref{scprod}), the scalar products of these vector fields are given by 
\begin{align}
\label{scpr-arc}
  g(\mathbf L,\mathbf N) = 
  g(\mathbf W,\mathbf W) = 1,\quad
  g(\mathbf L,\mathbf L) = 
  g(\mathbf L,\mathbf W) = 
  g(\mathbf N,\mathbf N) = 
  g(\mathbf N,\mathbf W) = 0. 
\end{align}
These vector fields satisfy the following system of differential equations  
\begin{align}
\label{Fren-null}
  \nabla_{\alpha'}\mathbf L = \mathbf W,\quad
  \nabla_{\alpha'}\mathbf N = \tau \mathbf W,\quad
  \nabla_{\alpha'}\mathbf W = \null -\tau \mathbf L - \mathbf N.
\end{align}
The frame $(\mathbf L, \mathbf N, \mathbf W)$ will be called the Frenet frame of the null curve $\alpha$, and the function $\tau$ will be called its pseudo-torsion.

\begin{theorem}
\label{theo3}
Let $\alpha$ be a null, non-degenerate and pseudo-arc length parametrized curve in a 3-dimensional oriented Lorentzian manifold. Then, the parametrization of $\alpha$ is equi-affine arc length and the Cartan frame $(\mathbf e_1, \mathbf e_2, \mathbf e_3)$ of $\alpha$ can be expressed with the help of the Frenet frame $(\mathbf L, \mathbf N, \mathbf W)$ as it follows
\begin{align}
\label{e123-lor}
  \mathbf e_1 = \mathbf L,\quad
	\mathbf e_2 = \mathbf W, \quad 
	\mathbf e_3 = \null- \tau \mathbf L - \mathbf N.
\end{align}
Moreover, the equi-affine curvatures $\varkappa_1$, $\varkappa_2$ of $\alpha$ can be expressed with the help of the pseudo-torsion $\tau$ as it follows
\begin{align}
\label{kap12-lor}
  \varkappa_1 = - \tau',\quad 
	\varkappa_2 = - 2 \tau.
\end{align}
\end{theorem}

\begin{proof}
At first, using (\ref{LWN}), we claim that 
\begin{align*}
  \varOmega(\mathbf L, \mathbf N, \mathbf W) 
	= \varOmega\left(\alpha',  \nabla_{\alpha'}\alpha', \nabla^2_{\alpha'}\alpha'\right). 
\end{align*}
On the other hand, it is a straightforward verification that $\varOmega(\mathbf L, \mathbf N, \mathbf W) = \pm1$. Consequently the frames $(\mathbf L, \mathbf N, \mathbf W)$ and $\left(\alpha', \nabla_{\alpha'}\alpha', \nabla^2_{\alpha'}\alpha'\right)$ have the same orientation and the pseudo-arc length parametrization of $\alpha$ is equi-affine arc length ($\varphi = 1$) since 
\begin{align*}
  \varOmega\left(\alpha', \nabla_{\alpha'}\alpha', 
	          \nabla^2_{\alpha'}\alpha'\right) = \varepsilon.
\end{align*}

The relations (\ref{e123-lor}) follow from (\ref{defCartfr}) by virtue of (\ref{LWN}) and $\varphi = 1$. Next, having (\ref{e123-lor}) and (\ref{Fren-null}), we compute 
\begin{align}
\label{nabe3}
  \nabla_{\alpha'}\mathbf e_3 
	= \null- \tau' \mathbf L - \tau \nabla_{\alpha'}\mathbf L - \nabla_{\alpha'}\mathbf N 
	= \null- \tau' \mathbf L - 2 \tau \mathbf W.
\end{align}
Now, the formulas (\ref{kap12-lor}) follow from (\ref{kp1}) and (\ref{kp2}) when using (\ref{e123-lor}) and (\ref{nabe3}) and knowing that $\varOmega(\mathbf L, \mathbf N, \mathbf W) = \varepsilon$ and $\varphi=1$. Precisely, we have 
\begin{align*}
  \varkappa_1 
	&= \varepsilon \varphi \varOmega(\nabla_{\alpha'}\mathbf e_3, \mathbf e_2, \mathbf e_3) 
	 = \varepsilon \varOmega(\null- \tau' \mathbf L - 2 \tau \mathbf W, \mathbf W, 
	                         \null- \tau \mathbf L - \mathbf N) \\[+3pt]
  &= - \varepsilon \tau' \varOmega(\mathbf L, \mathbf N, \mathbf W) 
	 = - \tau', \\[+3pt]
	\varkappa_2 
	&= \varepsilon \varphi \varOmega(\mathbf e_1, \nabla_{\alpha'}\mathbf e_3, \mathbf e_3) 
	 = \varepsilon \varOmega(\mathbf L, \null- \tau' \mathbf L - 2 \tau \mathbf W, 
	                             \null- \tau \mathbf L - \mathbf N) \\[+3pt]
	&= - 2 \varepsilon \tau \varOmega(\mathbf L, \mathbf N, \mathbf W) 
	 = - 2\tau.
\end{align*}
\end{proof}

We finish with the remark establishing a connection between equi-affine curvatures and the Schwarzian derivative of a certain function for null curves in the Minkowski spacetime $\mathbb E^3_1$. We describe it in the following way.

\begin{remark}
It is proved by one of the authors in \cite{Ol} (cf. also \cite{NP}) that any null and non-degenerate curve $\alpha$ in the 3-dimensional Minkowski spacetime $\mathbb E^3_1$ can be pseudo-arc length parametrized in the following way
\begin{align*}
  \alpha(t) = \alpha(t_0)
     + \frac{\varepsilon}{2} \int_{t_0}^t \frac{1}{f'(u)}
            \left(2f(u),f^2(u) - 1,f^2(u) + 1\right)\,dt,\ t,t_0\in I,
\end{align*}
where $f$ is a non-zero function with non-zero derivative $f'$ on $I$. Then the pseudo-torsion $\tau$ of $\alpha$ equals the Schwarzian derivative $S(f)$ of the function $f$, that is,
\begin{align*}
  \tau = S(f) = \left(\frac{f''}{f'}\right)^{\!\prime}
           - \frac12 \left(\frac{f''}{f'}\right)^2.
\end{align*}
Consequently, by Theorem \ref{theo3}, the equi-affine curvatures of $\alpha$ are
\begin{align*}
  \varkappa_1 = - (S(f))', \quad 
	\varkappa_2 = -2 S(f).
\end{align*}
\end{remark}


\end{document}